\documentclass{amsart}
\usepackage{amssymb}
\usepackage{amscd}

\newcommand{\N}{\mathcal{N}} \newcommand{\C}{\mathbf{C}}
\newcommand{\G}{\mathbf{G}} 
\newcommand{\Q}{\mathbf{Q}} 
\newcommand{\F}{\mathbf{F}} 
\newcommand{\Z}{\mathbf{Z}}

\newcommand{\M}{\mathcal{M}}

\renewcommand{\th}{\theta}

\newcommand{\et}{\mathrm{\acute{e}t}}

\newcommand{\sep}{\mathrm{sep}}
\newcommand{\perf}{\mathrm{perf}}

\newcommand{\eff}{\mathrm{eff}}

\newcommand{\isomto}{\overset{\sim}{\rightarrow}}

\DeclareMathOperator{\Sym}{Sym} 
\DeclareMathOperator{\Lie}{Lie}
\DeclareMathOperator{\Frob}{Frob}

\DeclareMathOperator{\GL}{GL}

\DeclareMathOperator{\Gal}{Gal}

\DeclareMathOperator{\Hom}{Hom}

\DeclareMathOperator{\Spec}{Spec}

\newcommand{\tM}{{t\mathcal{M}}}

\theoremstyle{plain}
\newtheorem{theorem}{Theorem}

\newtheorem{proposition}{Proposition}
\newtheorem{conjecture}{Conjecture}

\theoremstyle{definition}
\newtheorem{definition}{Definition}
\newtheorem{example}{Example}

\newtheorem{remark}{Remark}

\newtheorem{question}{Question}

\begin{document}

\markboth{Lenny Taelman}
 {Special Values and $t$-Motives: a Conjecture}

\title
 {Special $L$-Values of $t$-Motives: a Conjecture}

\author{Lenny Taelman}

\begin{abstract}
We propose a conjecture on special values of $ L $-functions
in a function field context with positive characteristic coefficients.

For $ M $ a uniformizable $ t $-motive
with everywhere good reduction we conjecture a relation between
the value of the Goss $ L $-function
$ L( M^\vee\!,\, s ) $ at $ s = 0 $ and the uniformization
of the abelian $ t $-module associated with $ M $.

When $ M $ is a power of the Carlitz $ t $-motive
the conjecture specializes to a theorem of Anderson and
Thakur on Carlitz zeta values. Beyond this case we present
numerical evidence.
\end{abstract}

\maketitle

\section{Introduction: Three flavors of special values}

Of the three flavors of special values of $ L $-functions that
are to be discussed now, only the third is logically relevant
to the rest of the paper. The first two are here to provide
some context.

\subsection{Number field base, characteristic zero coefficients}
\label{flavor1}

Let $ K $ be a number field and $ \bar{K} $ an algebraic closure
of $ K $.

\begin{definition}[see \cite{Serre68}]
A \emph{strictly compatible system} of $ \ell $-adic
representations of $ \Gal( \bar{K} / K ) $ is a collection
$ \rho = ( \rho_\ell )_\ell $ of continuous homomorphisms
$ \rho_\ell: \Gal( \bar{K} / K ) \to \GL( V_\ell ) $, one for every
rational prime $ \ell $, where the $ V_\ell $ are finite dimensional
$ \Q_\ell $-vector spaces, and subject to the condition that there exists
a finite set $ S $ of places of $ K $ such that:
\begin{enumerate}

	\item for all places $ v \notin S $ and for all $ \ell $
	coprime with $ v $ the representation $ \rho_\ell $ is
	unramified at $ v $;

	\item for such $ \ell $ and $ v $ the characteristic
	polynomial of $ \rho_\ell( \Frob_v ) $ has rational
	coefficients and does not depend on $ \ell $.

\end{enumerate}
\end{definition}

A natural source of strictly compatible systems is $ \ell $-adic
cohomology: consider the system $ ( V_\ell )_\ell $ where
$ V_\ell := H^i ( X_{ \bar{K}, \et }, \Z_\ell ) \otimes_{\Z} \Q $ are the $ \ell $-adic
cohomology groups of a smooth and projective variety $ X $ over $ K $,
or their Tate twists, and even subquotients of these constructed using
correspondences defined over $ K $. By \cite{Deligne74} these form
strictly compatible systems. Any system of representations that is
isomorphic to such a system is said to \emph{come from geometry}.

To any strictly compatible system $ \rho $ coming from geometry
and any finite set of
places $ S $ as above one associates an $ L $-function as follows.
First define for every finite place $ v $ that is not in $ S $ the polynomial 
\[
	P_v ( X ) := \det \left( 1 - X \rho_\ell ( \Frob_v ) \right)
	\in \Q[ X ]
\]
using any $ \ell $ which is coprime with $ v $. Then define the
$ L $-function of $ \rho $ away from $ S $ by the Euler product
\begin{equation}\label{Ldef}
	L_S ( \rho, s ) := \prod_{ v \notin S }
	P_v ( Nv^{-s} )^{ -1 }
\end{equation}
where the product ranges over all finite places of $ K $ not in $ S $
and where $ Nv \in \Z_{>0} $ denotes the norm of the place $ v $.
By \cite{Deligne74} this converges to a complex analytic
function for $ \Re( s ) $ sufficiently large. 

For any $ \rho $ coming from geometry and $ n \in \Z $ such
that $ L( \rho, s ) $ can be holomorphically continued to a neighborhood
of $ s = n $ we say that the complex number $ L( \rho, n ) $ is a
\emph{special value}. More generally, if $ L( \rho, s ) $ can be
meromorphically continued to a neighborhood of $ s = n $
we also call the leading coefficient of the Laurent series expansion of
$ L( \rho, s ) $ around $ s = n $ a special value.

There is a large zoo of theorems and conjectures concerning these
special values: Euler's $ \zeta( 2 ) = \pi^2 / 6 $, the class number
formula, the Birch and Swinnerton-Dyer conjecture, to name just a few.
A very general conjecture due to Beilinson \cite{Beilinson84}
and reformulated by Scholl \cite{Scholl91}
expresses all special values (up to a rational factor) in terms of
periods of mixed motives. (\emph{see also} the excellent survey
\cite{Kontsevich01}.)

\subsection{Function field base, characteristic zero coefficients}

Of course the above definition of an $ L $-function associated to
a strictly compatible system of $ \ell $-adic representations makes
perfect sense if $ K $ is not a number field but the function field
of a curve over a finite field with $ q $ elements.

Only the relation between special values and periods disappears from
the picture, because if $ \rho $ comes from geometry then there exists
a rational function $ f \in \Q( T ) $ such that
$ L( \rho, s ) = f ( q^{-s} ) $. In particular: if $ s = n $ is not
a pole of $ L( \rho, s ) $ then $ L( \rho, n ) \in \Q $.

(The interpretation of this rational number in terms of arithmetic
geometry and algebraic $ K $-theory is a very interesting problem
\cite{Schneider82} \cite{Milne86}, but it is not
the topic of this note.)

\subsection{Function field base, characteristic $p$ coefficients}

After having discussed the two flavors that we will \emph{not} be
concerned with, we now come to the central topic of this paper.

Let us start with an example of a special value of this third flavor.
Let $ A := \F_q[ t ] $ be the polynomial ring in one variable $ t $
over a finite field $ \F_q $ of $ q $ elements. Write $ A_+ $ for
the set of monic elements of $ A $.
The infinite sum
 \[ \zeta( n ) := \sum_{ f \in A_+ } f^{-n} \]
converges in $ \F_q(( t^{-1} )) $ for every $ n \in \Z_{>0} $.
For example, if $ q = 2 $ then one easily computes by hand
\[
	\zeta( 1 ) \in 1 + t^{-2} + t^{-3} + t^{-4} \F_2[[ t^{-1} ]].
\]
Using unique factorization in $ A $ we obtain
an expression as an infinite convergent Euler product:
 \[ \zeta( n ) = \prod_{ f } ( 1 - f^{-n} )^{-1}, \]
where the product runs over the monic irreducible elements.
These $\zeta( n )$ with $ n > 0 $ are examples of special values
about which our conjecture will say something. In fact, for
these examples the conjecture specializes to a theorem due to
Anderson and Thakur \cite{Anderson90}.

Now we generalize this example and turn to strictly compatible
systems of Galois representations. For every
non-zero prime ideal $ \lambda \subset \F_q[ t ] $ consider the
$ \lambda $-adic completion $ \F_q( t )_\lambda $ of
$ \F_q( t ) $.
Let $ K $ be a finite separable extension of $ \F_q( t ) $.
Let $ \rho = ( \rho_\lambda ) $ be a family of representations
of $ \Gal ( K^\sep / K ) $ on finite dimensional 
$ \F_q( t )_\lambda $-vector spaces, one for each prime ideal
$ \lambda $ of $ \F_q[ t ] $.
We call $ \rho $ a \emph{strictly compatible system} if there exists
a finite set $ S $ of places of $ K $ such that
\begin{enumerate}

	\item for every finite place $ v \notin S $ and for all
	$ \lambda $ not under $ v $ the representation
	$ \rho_\lambda $ is unramified at $ v $;

	\item for these $ \lambda $ and $ v $ the characteristic
	polynomial of Frobenius at $ v $ has coefficients in
	$ \F_q ( t ) $ and is independent of $ \lambda $.

\end{enumerate}

For every finite place $ v $ of $ K $ define $ \N v \in \F_q[ t ] $
to be a monic generator of the norm from $ K $ to $ \F_q( t ) $
of the ideal corresponding to $ v $. 

Now by analogy with (\ref{Ldef}) we define for every finite 
$ v \notin S $
\[
	P_v( X ) := \det( 1 - X \rho_\lambda ( \Frob_v ) )
	\in \F_q( t )[ X ]
\]
using any $ \lambda $ not below $ v $ and
\[
	L_S( \rho, n ) := \prod_{ v \notin S }
	P_v( \N v^{ -n } )^{ -1 },
\]
the product being over the finite places $ v $ that are not in $ S $.
This converges to an element of $ \F_q(( t^{-1} )) $ for all
sufficiently large \emph{integers} $ n $.

For example, if $ K $ is $ \F_q( t ) $, and $ \rho $ the family of trivial
representations then with $ S = \emptyset $ we have
\[
	L( \rho, n ) = \zeta( n ).
\]

In this context, a natural source of strictly compatible systems are 
$t$-motives, and our conjecture will have something to say about
the special value $ L( \rho, n ) $ provided that $ \rho $ comes from
a uniformizable $ t $-motive with everywhere good reduction.
(These notions will be explained in \S \ref{preliminaries}.)

To demand that $ \rho $ comes from a uniformizable $ t $-motive
is very natural, but the condition that it has everywhere good
reduction (which is equivalent with saying that $ S $ can be taken
to consist of only ``infinite'' places of $ K $ ) is an ugly condition
that should eventually be
removed. Unfortunately at present there are almost no examples
with bad reduction where the numerical data allows us to
make reasonable conjectures.

\begin{remark}
We speak about a ``special value'' $ L( \rho, n ) $, but
we have not defined $ L( \rho, s ) $ for any non-integral argument
$ s $. Goss \cite{Goss92} has shown that there is in fact an
analytic function $ L( \rho, s ) $ of which the $ L( \rho, n ) $ are
particular values (the tricky part is defining the domain of such a
function). We will not use this.
\end{remark}

Finally, we should point out that in a recent preprint
of Vincent Lafforgue \cite{Lafforgue08} formulas for
certain classes of special values in terms
of extensions of shtukas have been proven. We hope to discuss the
precise relation between his Theorems and our Conjectures in
a future paper.

\section{Preliminaries}
\label{preliminaries}

\subsection{Base and coefficients, notation}

To produce strictly compatible systems of Galois representations
over function fields it is very useful to separate the
base field from the coefficient rings. So we will look at
representations of $ \Gal( K^\sep / K ) $ with $ K $ a
function field containing $ \F_q $ on vector spaces
over completions $ \F_q( t )_\lambda $ of the a priori unrelated
rational function field $ \F_q( t ) $.

Eventually, to have a meaningful notion of $ L $-functions
we will fix an injective morphism $ \F_q( t ) \to K $, but we
will \emph{not} identify $ \F_q( t ) $ with the image.

(Such separation is impossible in the number field case,
in the same way that trying to adapt Weil's intersection-theoretical
proof of the Riemann Hypothesis for curves $ X $ over finite fields
to $ \Spec( \Z ) $ breaks down in the first step: the construction of
the surface $ X \times X $.)

\subsection{Table of notation}

\begin{itemize}
\item $ \F_q $: a fixed field with $ q $ elements.
\item {\it Base rings:}
\begin{list}{\labelitemi}{\leftmargin=4em}
	\item[$ K_\infty := $] $ \F_q(( \th^{-1} )) $, the field of Laurent
	series in $ \th^{-1} $ over $ \F_q $;
	\item[$ K := $]  a subfield of $ K_\infty $ that has finite degree
	over $ \F_q( \th ) $;
	\item[$ O_K := $]  the integral closure of the polynomial ring
	$ \F_q[ \th ] $ inside $ K $;
	\item[$ \C_\infty := $] the completion of an algebraic closure
	of $ K_\infty $;
	\item[$ K^\sep := $] the separable closure of $ K $ in $ \C_\infty $;
	\item[$ K^\perf := $] the perfection of $ K $.
\end{list}
\item {\it Coefficient rings:}
\begin{list}{\labelitemi}{\leftmargin=4em}
	\item[$ A := $] $ \F_q[ t ] $, polynomial ring in a variable $ t $;
	\item[$ A_\lambda := $] $ \varprojlim_n A / \lambda^n $, the
		$ \lambda $-adic completion of $ A $, where $ \lambda $
		is a non-zero prime ideal of $ A $;
	\item[$ F := $] $ \F_q( t ) $, the fraction field of $ A $;
	\item[$ F_\lambda := $] $ A_\lambda \otimes_A F $;
	\item[$ F_\infty := $] $ \F_q (( 1/t )) $.
\end{list}
\item {\it Relation between base and coefficients:} 
	$ i : \F_q[ t ] \to K $: the $ \F_q $-algebra
	homomorphism that maps $ t $ to $ \th $.
\end{itemize}

(The classical counterpart to this last map 
is the canonical morphism from $ \Z $ to any commutative ring.)

\subsection{$ t $-motives and Galois representations}

Let $ R $ be a commutative ring containing $ \F_q $.

\begin{definition}
A \emph{$ \sigma $-module} of \emph{rank} $ r $
over $ R $ is a pair $ ( M, \sigma ) $ of a projective
$ R \otimes _{\F_q} A $-module $ M $
of rank $ r $ and a map $ \sigma: M \to M $ such that
\begin{enumerate}
	\item $ \sigma $ is $ A $-linear;
	\item $ \sigma ( x m ) = x^q \sigma( m ) $ for all
		$ x \in R $ and $ m \in M $.
\end{enumerate}
\end{definition}

A morphism from $ ( M_1, \sigma_1 ) $ to $ ( M_2, \sigma_2 ) $
is a homomorphism $ f:  M_1 \to M_2 $ of
$ R \otimes_{\F_q} A $-modules such that
$ \sigma_2 \circ f = f \circ \sigma_1 $. 

We will often suppress the $ \sigma $ from the notation and write 
$ M $ for a $ \sigma $-module $ ( M, \sigma ) $.

If $ ( M_1, \sigma_1 ) $ and $ ( M_2, \sigma_2 ) $ are $ \sigma $-modules
then we define their tensor product to be the $ \sigma $-module
$ ( M_1 \otimes_{R\otimes A} M_2, \sigma_1 \otimes \sigma_2 ) $.
Similarly one can define symmetric and exterior powers. In particular,
given a $ \sigma $-module $ M $ one can consider its determinant
$ \det( M ) $ which is a $ \sigma $-module of rank one.

If $ R \to S $ is an $ \F_q $-algebra homomorphism and $ M $ a
$ \sigma $-module over $ R $ then we denote by $ M_S $ the
$ \sigma $-module over $ S $ obtained by extension of scalars:
\[
	M_S = ( M\otimes_R S, m\otimes s \mapsto \sigma(m) \otimes s^q ).
\]

\begin{definition}
A $ \sigma $-module $ ( M, \sigma ) $ over a field $ L $ is said
to be \emph{non-degenerate} if $ \det( \sigma ) : \det( M ) \to \det( M ) $
is non-zero. A $ \sigma $-module $ M $ over $ R $ is said to be
non-degenerate if $ M_L $ is non-degenerate for all $ R $-fields
$ R \to L $.
\end{definition}

Let $ M $ be a $ \sigma $-module over $ K $. For every non-zero
prime ideal $ \lambda $ of $ A $ we have that
\[
	T_\lambda( M ) := \varprojlim_n
	\left( M_{ K^\sep } / \lambda^n M_{ K^\sep } \right)^{\sigma=1}
\]
is naturally an $ A_\lambda $-module with a continuous action of
$ \Gal( K^\sep / K ) $ and
\[
	V_\lambda( M ) :=
	T_\lambda( M ) \otimes_{ A_\lambda } F_\lambda
\]
is naturally an $ F_\lambda $-vector space with a continuous action
of $ \Gal( K^\sep / K ) $. In general these need not be finitely
generated, yet one easily verifies:

\begin{proposition}
	If $ \sigma $ is non-degenerate
	then for all but finitely many $ \lambda $ 
	the dimension of $ V_\lambda( M ) $ equals the
	rank of $ M $. \qed
\end{proposition}

So far we have not used $ i $ which relates the base
and the coefficients. Recall that $ \th = i( t ) $.

\begin{definition}
An \emph{effective $ t $-motive} over $ K $ is a non-degenerate
$ \sigma $-module $ M $ over $ K $
such that $ \det ( M ) $ is isomorphic with the $ \sigma $-module
$ ( K[t]e, e \mapsto \alpha(t-\th)^ne ) $ for some 
$ \alpha \in K^\times $ and $ n \geq 0 $.
\end{definition}

The family of Galois representations associated with
an effective $ t $-motive forms a strictly compatible system:

\begin{proposition}[Thm 3.3 of \cite{Gardeyn01}]\label{strcmp}
Let $ M $ be an effective $ t $-motive over $ K $ of rank $ r $.
Then $ \dim V_\lambda( M ) = r $ for all $ \lambda $. Moreover,
there exists a finite set $ S $ of places of $ K $
such that
\begin{enumerate}
	
	\item for every place $ v \notin S $ and for all
	non-zero prime ideals $ \lambda $ coprime with
	$ i^\ast v $
	the representation $ V_\lambda( M ) $ is unramified
	at $ v $;
	
	\item for these $ \lambda $ and $ v $ the
	characteristic polynomial of Frobenius at $ v $
	has coefficients in $ A $ and is independent of $ \lambda $.
\qed
\end{enumerate}
\end{proposition}

\begin{example}
Let $ C $ be the \emph{Carlitz} $ t $-motive over $ K $. This is the rank
one effective $ t $-motive given by
\[
	C = ( K[ t ] e, e \mapsto ( t - \th ) e ).
\]
Let $ v $ be a finite place of $ K $ (\emph{i.e.} $ v $ does not lie
above the place $ \th = \infty $ of $ \F_q( \th ) $.) 
Let $ f \in \F_q[ \theta ] $ be
a monic generator of the ideal in $ \F_q[ \theta ] $ corresponding to
the norm of $ v $ in $ \F_q( \theta ) \subset K $. One verifies 
that
\begin{enumerate}
	\item the representation $ V_\lambda( C ) $ is unramified at
		$ v $ for all $ \lambda $ coprime with $ i^\ast v $;
	\item for such $ \lambda $ we have that $ \Frob_v $ acts
		as $ f( t )^{ -1 } \in \F_q( t ) $.
\end{enumerate}
So $ C $ plays the role of the Lefschetz motive $ \Q( -1 ) $.
\end{example}

\subsection{Abelian $ t $-modules}

Denote by $ K[ \tau ] $ the ring whose elements are polynomial 
expressions $ \sum a_i \tau^i $ with $ a_i \in K $ and where
multiplication is defined through
the rule $ \tau a = a^q \tau $ for $ a \in K $. The ring $ K[ \tau ] $
is canonically isomorphic with the endomorphism ring of the
$\F_q$-vector space scheme $ \G_a $ over $ K $.

If $ ( M, \sigma ) $ is an effective $ t $-motive over $ K $ then $ M $ is
naturally a left $ K[ \tau ] $ module through $ \tau m := \sigma(m) $.
Now consider the functor
\[
	E_M :
	\{  K \text{-algebras} \} \to 
	\{ A \text{-modules} \} :
	R \mapsto \Hom_{ K[\tau] }( M, R ),
\]
where $ R $ is a left $ K[ \tau ] $-module through $ \tau r := r^q $.
This functor is representable by an affine $ A $-module scheme.

Conversely, given an $ A $-module scheme $ E $ over $ K $ define 
\[
	M_E := \Hom_{K-\text{gr.sch.}}( E, \G_a ),
\]
which is naturally a left $ A \otimes_{\F_q} K[\tau] $-module.

\begin{theorem}[\S 1 of \cite{Anderson86}, \S 10 of \cite{Stalder07}]
The functors $ M \mapsto E_M $ and $ E \mapsto M_E $ form a pair of
quasi-inverse anti-equivalences between the categories of 
effective $ t $-motives $ M $ over
$ K $ that are finitely generated as left $ K[ \tau ] $-modules
and the category of $ A $-module schemes $ E $ over $ K $ that satisfy
\begin{enumerate}
	\item for some $ d \geq 0 $ the group schemes
	$ E_{K^\perf} $ and $ \G_{a,K^\perf}^d $ are isomorphic;
	\item $ t - \th $ acts nilpotently on $ \Lie( E ) $;
	\item $ M_E $ is finitely generated as a $ K[ t ] $-module. \qed
\end{enumerate}
\end{theorem}

\begin{definition}
An $ \F_q[t] $-module scheme $ E $ satisfying the above three conditions
is called an \emph{abelian $ t $-module} of dimension $ d $. An 
abelian $ t $-module of dimension one is called a \emph{Drinfeld module}.
\end{definition}

\begin{question}
Is the underlying group scheme of an abelian $ t $-module
isomorphic to $ \G_{a}^d $ over $ K $?
\end{question}

For Drinfeld modules this is indeed the case, since the
only form of $ \G_{a} $ that has infinite endomorphism ring is 
$ \G_a $ itself (\emph{see} \cite{Russel70}, \emph{see also}
\S 10 of \cite{Stalder07}).

The tangent space at the identity of $ E $ can be expressed
in terms of $ M_E $ as follows:
\begin{proposition}[\emph{see} \cite{Anderson86}]
$ \Lie_E( K ) = \Hom_K( M_E/K\sigma(M_E), K ) $.
\end{proposition}

Also the Galois representations associated with $ M_E $
can be expressed in terms of $ E $. If $ \lambda = ( f ) \subset A $
a non-zero prime ideal then define the $ \lambda $-adic
Tate module of $ E $ to be
\[
	V_\lambda( E ) :=
	( \varprojlim_n E[ f^n ]( K^\sep ) )
	\otimes_{ A_\lambda } F_\lambda.
\]
If $ M $ is the effective $ t $-motive associated with $ E $ then we have
\begin{proposition}
$ V_\lambda( M_E ) \cong \Hom ( V_\lambda( E ), F_\lambda ) $. \qed
\end{proposition}

\subsection{Uniformization}

\begin{proposition}[\emph{see} \S 2 of \cite{Anderson86}]
Let $ E $ be an abelian $ t $-module over $ K $.
\begin{enumerate}
\item There exists a unique entire $ A $-module homomorphism
$ \exp_E : \Lie_E( \C_\infty ) \to E( \C_\infty ) $ that
is tangent to the identity;
\item The kernel of $ \exp_E $ is a finitely
generated free discrete sub-$ A $-module in $ \Lie_E( \C_\infty ) $.
\end{enumerate}
\end{proposition}

When $ \exp_E $ is surjective this yields an analytic description
of the $ A $-module $ E( \C_\infty ) $ as the quotient of 
$ \Lie_E( \C_\infty ) $ by a discrete submodule. 

Denote by $ M $ the $ t $-motive associated with $ E $.
The following theorem characterizes the $ E $ such that $ \exp_E $
is surjective:

\begin{theorem}\label{unifchar}
The following are equivalent:
\begin{enumerate}
	\item $ \exp_E $ is surjective;
	\item the rank of $ \ker \exp_E $ equals the rank of $ M $;
	\item for all $ \lambda $ the restriction of the
	Galois representation
	$ \rho_\lambda: \Gal( K^\sep / K ) \to \GL( V_\lambda( M ) ) $
	to
	$ \Gal( K_\infty^\sep / K_\infty ) $ has finite image.
\end{enumerate}
\end{theorem}

When these equivalent statements hold we say that $ E $ (or $ M $)
is \emph{uniformizable}.

\begin{proof}
The equivalence of (i) and (ii) is part of Theorem 4 of \cite{Anderson86},
the equivalence of (iii) and (i) is part of Theorem 5.12 of \cite{Gardeyn01}. 
\end{proof}

Examples of uniformizable effective $t$-motives are provided by the
following:

\begin{proposition}
\begin{enumerate}
\item Drinfeld modules are uniformizable;
\item The tensor product of two uniformizable
effective $t$-motives is uniformizable;
\item Subquotients of uniformizable effective $t$-motives are uniformizable.
\end{enumerate}
\end{proposition}

\begin{proof}
The first claim is shown in \cite{Drinfeld74E}. The other two
follow at once from the third characterization in Theorem \ref{unifchar}.
\end{proof}

\subsection{Good reduction}

Let $ M $ be an effective $ t $-motive over $ K $.

\begin{theorem}\label{goodred}
The following are equivalent:
\begin{enumerate}
	\item there exists a non-degenerate
	$ \sigma $-module $ \M $ over $ O_K $
	and an isomorphism $ \alpha : \M_K \to M $;
	\item $ ( H_\lambda( M, \sigma ) )_\lambda $ forms a strictly
	compatible system with exceptional set $ S $ consisting uniquely
	of infinite places of $ K $.
\end{enumerate}
Moreover, if it exists the pair $ ( \M, \alpha ) $ is unique
up to a unique isomorphism.
\end{theorem}

If these equivalent statements hold we say that $ M $ has
\emph{everywhere good reduction} and we call $ \M $ a 
\emph{good model} for $ M $.

\begin{proof}
This follows easily from Theorem 1.1 of \cite{Gardeyn02}.
\end{proof}

\subsection{The $ L $-function of an effective $ t $-motive}

Let $ M $ be an effective $ t $-motive over $ K $. Let $ S $ be an exceptional
set of places of $ K $ for the strictly compatible system of
Galois representations $ \rho = ( \rho_\lambda )_\lambda $ associated
with $ M $. 

Let $ v $ be a finite place of $ K $ corresponding to a prime ideal
$ I \subset O_K $. Denote by $ \N v \in A $ the unique monic generator
of the inverse image image under $ i : A \to \F_q[\th] $
of the norm of $ I $ in $ \F_q[ \th ] $. 

For any finite $ v $ that is not in $ S $ define
\[
        P_v( X ) := \det( 1 - X \rho_\lambda ( \Frob_v ) )
        \in A[ X ]
\]
using any $ \lambda $ such that $ i( \lambda ) $ is coprime with $ v $ and
\[
        L_S( M, n ) := \prod_{ v \notin S }
        P_v( \N v^{ -n } )^{ -1 },
\]
the product being over the finite places $ v $ that are not in $ S $.
This converges to an element of $ F_\infty $ for all
sufficiently large integers $ n $.

Now assume (for simplicity) that $ M $ has everywhere good
reduction and let $ \M $ be a model for $ M $ over $ O_K $.
Let $ v $ be a finite place of $ K $ and $ k(v) $ the residue class
field of $ v $. Let $ d(v) $ be the degree of $ k(v) $ over $ \F_q $.
Note that $ \sigma^{d(v)} $ is a \emph{linear} endomorphism of
$ \M_{k(v)} $. The Euler factors in $ L( M, n ) $ can be computed
as follows:

\begin{proposition}\label{eulerfactor}
$ P_v( X ) = \det( 1 - X \sigma^{d(v)} | \M_{k(v)} ) $.\qed
\end{proposition}

\subsection{The $ L $-function of a $ t $-motive}

The category $ \tM_\eff $
of effective $ t $-motives over $ K $ with its tensor product
is an $ A $-linear tensor category, but it is not closed under
duals.

After formally inverting the object $ C $ it embeds into an $ A $-linear
\emph{rigid} tensor category $\tM$. The objects of this latter category
are called \emph{$ t $-motives}. They
are formal expressions $ M \otimes C^{\otimes n} $ with $ M $ a
$ t $-motive and $ n \in \Z $, and morphisms are defined as
\[
	\Hom_{\tM}( 
	M_1\otimes C^{\otimes n_1}, 
	M_2\otimes C^{\otimes n_2} )
	:=
	\Hom_{\tM_\eff}(
	M_1\otimes C^{\otimes n_1 + n},
	M_2\otimes C^{\otimes n_2 + n} )
\]
for $ n $ sufficiently large so that both $ n_1 + n $ and $ n_2 + n $
become non-negative. This is independent of $ n $ because
for every pair $ M_1, M_2 $ of effective $t$-motives there is a
canonical isomorphism
\[
	\Hom_{\tM_\eff}( M_1, M_2 ) =
	\Hom_{\tM_\eff}( M_1 \otimes C, M_2 \otimes C ).
\]
Given a $ t $-motive $ M $ there exists a dual $ t $-motive $ M^\vee $,
and the operations $ ( - )^\vee $ and $ \otimes $ satisfy all the usual
properties from representation theory.
(Proofs and more details can be found in \S 2 of \cite{Taelman08}.)

Since the functors
\[
	V_\lambda :
	\tM_\eff \to
	\{ \Gal( K^\sep / K ) \text{-representations} / F_\lambda \}
\]
respect the tensor product, they extend to $\tM$. 
In particular, this allows us to define $ L $-functions for
$t$-motives. 

We have that $ L( M \otimes C, n + 1 ) = L( M, n ) $, which allows
us to shift special values around in a way that is more or less
obvious when working with $ t $-motives but rather non-trivial
when working with abelian $ t $-modules. This is one of the reasons
that we consider $ t $-motives in this paper, rather than working
uniquely with abelian $ t $-modules. Another reason is given by
the notion of good reduction, which is quite straight-forward 
on the $ t $-motives side, but rather subtle on the abelian
$ t $-modules side.

\subsection{Convergence}

So far we have ignored questions of convergence. The following proposition
guarantees that the special values that occur in our conjecture will
be well-defined.

\begin{proposition}\label{convergence}
If $ M $ is an effective $ t $-motive over $ K $ that is finitely
generated as a $ K[ \tau ] $-module, then the Euler product
for $ L( M^\vee, 0 ) $ converges.
\end{proposition}

\begin{proof}
As one might expect, the proof is based upon bounds for the 
$ 1/t $-adic valuations of eigenvalues of Frobenius. 

Consider the $ K^\sep(( t^{-1} )) $-vector space
$ M(( t^{-1} )) := M \otimes_{ K[t] } K^\sep(( t^{-1} )) $. The action of
$ \sigma $ on $ M $ extends to action on $ M(( t^{-1} )) $ that is
$ F_\infty $-linear and satisfies $ \sigma( x m ) = x^q \sigma( m ) $ 
for all $ x \in K^\sep $ and $ m \in M(( t^{-1} )) $. Since $ \sigma $
is not linear it does not make sense to speak about eigenvalues of
$ \sigma $, yet by \cite[\S 5.1]{Taelman08} the valuations
$ \lambda_1, \lambda_2, \cdots, \lambda_r $ of the
eigenvalues of a matrix representing $ \sigma $ relative to a chosen
basis of $ M $ do not depend on the chosen basis. (In other words:
the newton polygon of the characteristic polynomial is a well-defined
invariant of $ ( M(( t^{-1} )), \sigma ) $.)

{\it Claim.} $ \lambda_i < 0 $ for all $ i $.

The finite generation of $ M $ as $ K[ \tau ] $-module guarantees
that there exists a finite dimensional $ K^\sep $-vector subspace
$ V \subset M(( t^{-1} )) $ such that
\begin{equation}\label{dieufin}
	\cup_{ i \geq 0 } \cup_{ j \geq 0 }
	t^{-i} K^\sep \sigma^j( V )
	\text{ is dense in }
	M (( t^{-1} )).
\end{equation}
But from the classification \cite[\S 5.1]{Taelman08} it follows that
there exists a positive integer $ n $ and a basis of
$ M(( t^{-1} )) $ such that the
action of $ \sigma^n $ with respect to that basis is given by
\begin{eqnarray*}
	\left( \begin{array}{cccc}
		t^{-n\lambda_1}  &  0  &  \cdots  &  0 \\
		0  &  t^{-n\lambda_2}  &  \cdots  &  0 \\
		\vdots  &  \vdots  &  &  \vdots \\
		0  &  0  &  \cdots  &  t^{-n\lambda_r} 
	\end{array} \right).
\end{eqnarray*}
Hence for (\ref{dieufin}) to hold with a finite dimensional $ V $
one needs that $ \lambda_i < 0 $ for all $ i $, which proves the claim.

To finish the proof it now suffices to observe that for almost all
places $ v $ the Newton polygon of $ \sigma $ does not change under
reduction mod $ v $.
\end{proof}

\begin{remark}
	The converse holds as well: if $ M $ is an effective $ t $-motive
over $ K $ then the Euler product defining $ L( M^\vee, 0 ) $ converges
if and only if $ M $ is finitely generated over $ K[ \tau ] $. This
follows essentially from \cite[Theorem 5.3.1]{Taelman08}.
\end{remark}

\section{The conjecture}

For a $ t $-module $ E $ over $ K $ define
\[
	W_E := \Lie_E / ( t - \th ) \Lie_E 
\]
and write $ w $ for the canonical projection $ \Lie_E \to W_E $.
Note that $ W_E( K_\infty ) $ carries naturally the structure
$ F_\infty $-vector space (coming from the action of $ A $)
as well as a that of a $ K_\infty $-vector space and that the
two structures coincide under the identification ``$ t = \th $.''

Now assume that the $ t $-motive $ M $ associated with $ E $ has
everywhere good reduction and let $ ( \M, \alpha: \M_K \isomto M ) $
be a good model. We define $ E( O_K ) \subset E( K ) $ to be the image of
the map
\[
	\Hom_{O_K[\tau]} ( \M, O_K ) \to \Hom_{K[\tau]} ( M, K ) = E( K )
\]
induced by $ \alpha $. Also we define $ \Lie_E( O_K ) $ as
the image of
\[
	\Hom_{O_K} ( \M / \sigma \M, O_K ) \to
	\Hom_K ( M / \sigma M, K ) = \Lie_E( K ).
\]
and $ W_E( O_K ) \subset W_E( K ) $ as the image of $ \Lie_E( O_K ) $
under $ w $.

\begin{conjecture}
Let $ E $ be a uniformizable abelian $ t $-module over $ K $ such that
the associated $ t $-motive $ M $ has has everywhere good reduction.

There exists a sub-$ A $-module $ Z \subset \Lie_E( K_\infty ) $ of rank
$ \dim W_E $ such that $ \exp_E ( Z ) \subset E( O_K ) $ and such that
\[
	\bigwedge^{\dim W_E}_{A} w( Z ) = 
	L( E, 0 ) \cdot
	\left( \bigwedge^{\dim W_E}_{A}  W_E ( O_K ) \right)
\]
as $ A $-lattices inside the $ 1 $-dimensional
$ F_\infty $-vector space
$ \bigwedge^{\dim W_E}_{ K_\infty } W_E( K_\infty ) $.
\end{conjecture}

\begin{remark} $ L( E, 0 ) = L( M^\vee, 0 ) $.
\end{remark}

\begin{theorem}[\cite{Anderson90}] \label{carlitzpower}
For $ M = C^{\otimes n} $
the conjecture holds.\qed
\end{theorem}

\begin{proposition} If the conjecture holds for $ M_1 $ and $ M_2 $
then it also holds for $ M_1 \oplus M_2 $. \qed
\end{proposition}

\section{Numerical experiments}

Given a $ t $-motive $ M $ and an $ n $ such that
the Euler product defining $ L( M, n ) $ converges
one can numerically approximate
\[
	L( M, n ) \in \F_q(( t^{-1} ))
\]
simply by multiplying all Euler factors at
places of degree $ \leq d $.
The proof of Proposition \ref{convergence} yields hard
error estimates for this approximation. This bound is linear
in $ d $ and hence this algorithm will compute
$ L( M, n ) $ modulo $ t^{-X} \F_q[[ t^{-1} ]] $ in a running
time that is exponential in $ X $.

Since the conjecture does not predict the module 
$ Z $, or does not even give bounds on the ``height'' of
generators of $ Z $, it does not lend itself to numerical
falsification. 
Yet we have systematically found that when working with $ M $ 
of low (naive) height there is always a $ Z $ of low (naive) height
for which the conjecture holds numerically to relatively high
precision. 

Before we state some of these numerical examples we introduce
the \emph{logarithm} of an abelian $ t $-module, which we will
need to produce candidate modules $ Z $ in some of these examples.

\subsection{The logarithm of an abelian $ t $-module}

Let $ E = ( \G_a^d, \phi ) $ be an Abelian $ t $-module over
$ K_\infty $. If we identify
$ \Lie_E( \C_\infty ) $ and $ E( \C_\infty ) $ with $ \C_\infty^d $ in
the obvious way then $ \exp_E : \Lie_E( \C_\infty ) \to E( \C_\infty ) $
can be expressed as a power series
\[
	\exp_E = \sum_{i=0}^{\infty} e_i \tau^i
\]
with $ e_i \in M_d( K_\infty ) $ and $ e_0 = 1 $. We claim that there
is a unique power series
\[
	\log_E = \sum_{i=0}^{\infty} l_i \tau^i
\]
with $ l_i \in M_d( K_\infty ) $ and $ l_0 = 1 $ such that
\begin{equation}\label{explog}
	\exp_E \log_E = 1.
\end{equation}
Indeed, if $ n > 0 $ then comparing  coefficients of $ \tau^n $
in (\ref{explog}) yields
\[
	l_n + e_1 \tau( l_{n-1} ) + \cdots + e_n \tau^n( l_0 ),
\]
where $ \tau( b ) $ is the matrix obtained from $ b $ by raising
every entry to the $ q $-th power. This last expression gives a recursion
for the $ l_i $ that shows that there is a unique power series $ \log_E $
satisfying (\ref{explog}).

Given an $ x \in E( K_\infty ) $ it is not necessarily true that
the infinite sum $ \log_E( x ) $ converges, but when it does
converge then clearly $ \exp_E( \log_E( x ) ) = x $. 

\subsection{$ L( E, 0 ) $ with $ E $ a Drinfeld module}

If $ E= ( E, \varphi ) $ is a Drinfeld module then
$ W_E= \Lie_E $ and hence one-dimensional. So if $ E $ 
has everywhere good reduction the conjecture predicts that
\begin{equation}\label{logalg}
	\exp_E ( L( E, 0 ) e ) \in E( O_K )
\end{equation}
where $ e \in \Lie_E( K_\infty ) $ is a generator defined over
$ O_K $.

If $ E $ has rank $ 1 $ over $ K = \F_q( \th ) $ and has
everywhere good reduction then it is necessarily of the form
\[
	E = ( \G_a, t \mapsto \th + \alpha\tau )
\]
with $ \alpha \in \F_q^\times $. We have
\[
	L( E, 0 ) = \sum_{f\in A_+} \frac{ \alpha^{\deg(f)} }{ f }
\]
and for these (\ref{logalg}) is known. (If $ \alpha = 1 $
this is Theorem \ref{carlitzpower} with $ n = 1 $. For other
values of $ \alpha $ one reduces to this case by a change of variable
$ t' := \alpha^{-1} t $.)

If the rank of $ E $ is higher than one and if $ E $ does not have
CM then the methods of the proof break down completely since there is
no explicit description of $ L( E, 0 ) $ as an infinite sum,
only as an Euler product.

However, $ L( E, 0 ) $ can be approximated numerically.

\begin{example}
Let $ q=2 $ and $ E = ( \G_a, t \mapsto \th + \tau + \tau^2 ) $
over $ K = \F_2( \th ) $. This Drinfeld module does not have
complex multiplication over $ K^\sep $.
We have
\begin{eqnarray*}
	L( E, 0 ) &\in&
	1 + t^{-2} + t^{-3} + t^{-5} + t^{-7} + t^{-9} + \\
	&& t^{-10} + t^{-17} + t^{-18} + t^{-19} \F_2[[ t^{-1} ]]
	\subset F_\infty.
\end{eqnarray*}
If we identify $ E( K ) = \G_a( K ) = K $ then one verifies that
$ E( O_K ) = O_K $. Using the natural generator
$ e \in \Lie_E( K_\infty ) $ we compute
\[
	\exp_E ( L( E, 0 )\, e ) \in
	1 + \th^{-19} \F_2[[ \th^{-1} ]] \subset K_\infty,
\]
so $ \exp_E ( L( E, 0 )\, e ) $ is at least very close to
an element of $ E( O_K ) $.
\end{example}

Similarly but now $ q = 3 $ and
$ E = ( \G_a, t \mapsto \th + \th\tau - \tau^2 ) $. We find that
\[
	\exp_E ( L( E, 0 )\, e ) \in
	1 + \th^{-12} \F_3[[ \th^{-1} ]].
\]

We have computed hundreds of such examples over $ \F_2( \th ) $,
$ \F_3( \th ) $ and $ \F_5( \th ) $ (but to a slightly lower precision than
the examples above), and in all of them $\exp_E( L( E, 0 ) e )$
coincided with a polynomial in $ \th $ (not always the constant
polynomial $ 1 $), within the computed precision.

\smallskip
Finally a rank $ 3 $ example:

\begin{example}
Take $ q = 2 $ and $ E = ( \G_a, t \mapsto \th + \tau + \tau^3 ) $. 
Then
\[
	\exp_E( L( E, 0 )\, e ) \in
	1 + \th^{-12} \F_2[[ \th^{-1} ]].
\]
\end{example}

\subsection{$ L( M, 2 ) $ with $ M $ the $ t $-motive of a
	Drinfeld module of rank $ 2 $}

Let $ E $ be a Drinfeld module of rank $ 2 $ and $ M = M( E ) $.
We have that $ M^\vee \cong M \otimes \det(M)^\vee $, so if we put
\[
	\tilde{ M } := M \otimes C^{\otimes 2} \otimes \det(M)^\vee
\]
then
\[
	L( M, 2 ) = L( \tilde{ M }^\vee, 0 ).
\]
Let $ \tilde{ E } $ be the $ t $-module corresponding to
$ \tilde{ M } $. Then $ \tilde{ E } $ has dimension $ 3 $ and the
maximal quotient $ w : \Lie_{\tilde{E}} \to W_{\tilde{E}} $ on which
$ t - \th $ acts trivially is two-dimensional. From the conjecture we
should therefore expect to express
$ L( M, 2 ) = L( \tilde{E}, 0 ) $ as a two by two determinant. Here is
an explicit example:

\begin{example}
Let $ q = 2 $ and $ E = ( \G_a, t \mapsto \th + \tau + \tau^2 ) $.
Then there is an $ O_K[ t ] $-basis for $ \M $ on which $ \sigma $
is expressed as
\[
	\left(
	\begin{array}{cc}
		1	& \th+t	\\
		1	& 0
	\end{array}
	\right).
\]
Note that $ \det( M ) = C $, so the action of $ \sigma $ on
the obvious basis for $ \tilde{ M } = M \otimes C $ is given by
\[
        \left(
        \begin{array}{cc}
                \th+t    & \th^2+t^2  \\
                \th+t    & 0
        \end{array}
        \right).
\]
From this the corresponding $ t $-module $ \tilde{ E } $ can be computed.
It is given by $ \tilde{ E } = ( \G_a^3, \varphi ) $, where $ \varphi $ is
determined by
\[
	\varphi(t) \left( \begin{array}{c}
	x_1 \\ x_2 \\ x_3
	\end{array} \right)
	=
	\left( \begin{array}{c}
	x_3 \\
	\th x_1 + \th x_2 + x_3 + \tau( x_1 ) \\
	\th^2 x_1 + \tau( x_2 )
	\end{array} \right).
\]
The quotient $ w : \Lie_{\tilde{F}} \to W_{\tilde{F}} $ takes the explicit
form
\[
	w \left( \begin{array}{c}
	\xi_1 \\ \xi_2 \\ \xi_3
	\end{array} \right)
	=
	\left( \begin{array}{c}
	\xi_1 + \xi_2 \\
	\th \xi_1 + \xi_3
	\end{array} \right)
\]
Now let $ z_1 = ( 1, 0, 0 ) $ and $ z_2 = ( 0, 0, 1 ) $ in
$ \tilde{E} ( O_K ) $. Then $ \log_{\tilde{E}} ( z_1 ) $ and
$ \log_{\tilde{E}} ( z_1 ) $ are well-defined elements of
$ \Lie_{\tilde{E}} ( K_\infty ) $ (the defining infinite sums
converge) and the ratio of the determinant
\[
	w( \log_{\tilde{E}} ( z_1 ) ) \wedge w( \log_{\tilde{E}} ( z_2 ) )
	\in \wedge^2 \Lie_E( K_\infty )
\]
with
\[
	L( \tilde{E}, 0 ) ( (1,0) \wedge (0,1) )
\]
is computed to lie in $ 1 + \th^{-31} \F_2[[ \th^{-1} ]] $.
So the conjecture seems to hold with $ Z $ the module generated by
$ \log_{\tilde{E}} ( z_1 ) $ and $ \log_{\tilde{E}} ( z_2 ) $.
\end{example}

\subsection{$ L( ( Sym^2 M )^\vee, 0 ) $ with $ M $ the $ t $-motive
	of a rank $ 2 $ Drinfeld module}

Let $ M $ be the $ t $-motive of a rank $ 2 $ Drinfeld module. Then
$ \Sym^2 M $ is the $ t $-motive of a rank $ 3 $ and dimension $ 3 $
$ t $-module $ E $. The quotient $ \Lie_E / ( t - \th ) \Lie_E $
is two-dimensional. 

\begin{example}
Let $ q = 3 $ and $ M $ the $ t $-motive of the Drinfeld module
$ ( \G_a, t \mapsto \th - \tau + \tau^2 ) $ over $ \F_3( \th ) $.
The action of $ \sigma $ on a suitable basis of $ \Sym^2 M $ is given by
\[
	\left(
	\begin{array}{ccc}
		1  &  t - \th  &  t^2 + \th t + \th^2 \\
		1  &  \th - t  &  0 \\
		1  &  0  &  0
	\end{array}
	\right).
\]
and the corresponding $ t $-module is
$ E = E_{\Sym^2 M} = ( \G_a^3, \varphi ) $, where
$ \varphi $ is given by
\[
	\varphi(t) \left( \begin{array}{c}
		x_1 \\ x_2 \\ x_3
	\end{array} \right)
	=
	\left( \begin{array}{c}
        	\th x_1 - x_1^q - x_3^q \\
		- \th x_1 - \th^2 x_3 - x_3^q + x_3^{ q^2 } \\
		x_1 + x_2 - \th x_3 
	\end{array} \right).
\]
The quotient $ w_E : \Lie_E \to W_E = \Lie_E / ( t - \th ) \Lie_E $ is
\[
	w \left( \begin{array}{c}
		\xi_1 \\ \xi_2 \\ \xi_3
	\end{array} \right)
	=
	\left( \begin{array}{c}
		\xi_1 \\ \xi_2 + \th \xi_3
	\end{array} \right).
\]
and we find that with $ Z $ the module generated by
$ \log_E ( 1, 0, 0 ) $ and $ \log_E ( 0, 1, 0 ) $ the
conjecture is compatible with the computed approximation
\begin{eqnarray*}
	L( E, 0 ) &\in& 
	1 + t^{-3} + t^{-5} + t^{-6} + t^{-7} - t^{-8} + t^{-11}
	- t^{-12} + t^{-13} \\
	&& - t^{-15} + t^{-16} - t^{-17} - t^{-18}
	+ t^{-19} + t^{-20} \F_2[[ t^{-1} ]].
\end{eqnarray*}
\end{example}

\section{A challenge}

Let $ f \in A $ be irreducible and
$ \chi : ( A / f )^\times \to \bar{\F_q}^\times $ be a group
homomorphism. Extend $ \chi $ to a multiplicative map
$ A \to \bar{\F_q} $ in the obvious way. Anderson
\cite{Anderson96} has given an expression for
\[
	L( \chi, 1 ) := \sum_{ f \in A_+ } \frac{ \chi(f) }{ f }
	\in \bar{ \F_q } (( 1/t ))
\]
in terms of Carlitz logarithms. So one can certainly say
something about some special values related to $ t $-motives
with bad reduction.

Yet here is a challenge: let $ E $ be the Drinfeld module
$ ( \G_a, t \mapsto \th + \th^{-1} \tau + \tau^2 ) $
over $ \F_2( \th ) $. Let $ v $ be the place $ \th = 0 $ of
bad reduction. Find an expression for
\begin{eqnarray*}
	L_{\{v,\infty\}}( E, 0 ) &\in& 
	1 + t^{-7} + t^{-9} + t^{-10} + t^{-11} + t^{-13} + \\
	&& t^{-14} + t^{-15} + t^{-17} + t^{-18} + 
	t^{-19} \F_2[[ t^{-1} ]].
\end{eqnarray*}

\bibliographystyle{plain}
\bibliography{../../master}

\end{document}